\documentclass[11pt]{amsart}
\usepackage[T1]{fontenc}
\usepackage[utf8]{inputenc}
\usepackage{amsmath,amssymb,amsthm,amsxtra,mathtools}
\usepackage[varg]{txfonts}
\usepackage{microtype}
\usepackage[dvipsnames,svgnames,table]{xcolor}
\usepackage{aliascnt}
\usepackage[linktocpage=true,colorlinks=true,linkcolor=Blue,
  citecolor=BrickRed,urlcolor=RoyalBlue]{hyperref}
\usepackage[nameinlink,noabbrev]{cleveref}

\hypersetup{
  pdfauthor={Aelson Sobral and Eduardo V. Teixeira},
  pdftitle={From qualitative smoothness to uniform estimates for fully nonlinear elliptic equations}
}

\numberwithin{equation}{section}

\theoremstyle{plain}

\newtheorem{theorem}{Theorem}[section]

\newaliascnt{proposition}{theorem}
\newtheorem{proposition}[proposition]{Proposition}
\aliascntresetthe{proposition}

\newaliascnt{lemma}{theorem}
\newtheorem{lemma}[lemma]{Lemma}
\aliascntresetthe{lemma}

\newaliascnt{corollary}{theorem}
\newtheorem{corollary}[corollary]{Corollary}
\aliascntresetthe{corollary}

\theoremstyle{definition}

\newaliascnt{definition}{theorem}
\newtheorem{definition}[definition]{Definition}
\aliascntresetthe{definition}

\theoremstyle{remark}

\newaliascnt{remark}{theorem}
\newtheorem{remark}[remark]{Remark}
\aliascntresetthe{remark}

\newcommand{\R}{\mathbb R}
\newcommand{\Rn}{\mathbb R^n}
\newcommand{\Sn}{\mathcal S^n}
\newcommand{\norm}[1]{\left\lVert #1\right\rVert}
\newcommand{\abs}[1]{\left\lvert #1\right\rvert}
\newcommand{\osc}{\operatorname*{osc}}
\newcommand{\dist}{\operatorname{dist}}
\newcommand{\Reg}{\operatorname{Reg}_2}
\newcommand{\Sing}{\operatorname{Sing}_2}
\newcommand{\Qual}{\alpha_{\mathrm{qual}}}
\newcommand{\Est}{\alpha_{\mathrm{est}}}

\title[Qualitative smoothness implies uniform estimates]
{From qualitative smoothness to uniform estimates for fully nonlinear elliptic equations}

\author[A. Sobral]{Aelson Sobral}
\address{Applied Mathematics and Computational Sciences (AMCS), Computer,
Electrical and Mathematical Sciences and Engineering Division (CEMSE),
King Abdullah University of Science and Technology (KAUST), Thuwal,
23955-6900, Kingdom of Saudi Arabia}
\email{aelson.sobral@kaust.edu.sa}

\author[E. V. Teixeira]{Eduardo V. Teixeira}
\address{Department of Mathematics, Oklahoma State University,
Stillwater, OK 74078, USA}
\email{eduardo.teixeira@okstate.edu}


\begin{document}

\begin{abstract}
We establish a general mechanism that turns qualitative smoothness into
uniform, \emph{quantitative} regularity estimates for fully nonlinear elliptic
equations. The central conclusion is that, for compact classes preserved by the natural rescalings of the equation, qualitative and quantitative regularity have the same critical threshold. At second order, mere
twice differentiability throughout the corresponding centered hull already
forces a uniform \(C^{2,\alpha}\)-theory for some \(\alpha>0\).  A one-sided
tangent version gives a new partial regularity criterion for a single
operator: every point at which a solution is twice differentiable is regular,
whereas the singular set has universal positive codimension.  The proof
combines compactness of the full renormalized class with a scale-adaptive
improvement of flatness, converting information available separately at each
profile into estimates uniform across the entire class.  The argument
requires no quantitative control of the assumed smoothness and opens a path
in problems where smoothness is visible, but estimates remain out of reach.
\end{abstract}

\maketitle
\section{Introduction}

The point of departure is an elementary observation from functional
analysis. Let $L_Au=\operatorname{tr}(AD^2u)$, where
$\lambda I\leq A\leq\Lambda I$, and consider the bounded solutions of
$L_Au=0$ in $B_1$, endowed with the supremum norm. This is a Banach space. If
every such solution belongs to $C^{1,\gamma}_{\mathrm{loc}}(B_1)$, the closed
graph theorem applied to $u\mapsto Du$ gives
\begin{equation}
 \|Du\|_{L^\infty(B_{1/2})}
 +[Du]_{C^{0,\gamma}(B_{1/2})}
 \le C_A \|u\|_{L^\infty(B_1)}.
 \label{eq:intro-linear-closed-graph}
\end{equation}
Thus, in the linear setting, mere membership in a smoothness class produces
an a priori estimate of Schauder type.

For a fully nonlinear equation there is no vector space of solutions, and
normalizing a solution usually changes the operator. The question is then:

\begin{quote}
\emph{When does qualitative smoothness of solutions yield a uniform,
scale-invariant estimate?}
\end{quote}

Here $\Sn$ denotes the space of real symmetric $n\times n$ matrices. Fix
$n\ge2$ and $0<\lambda\leq\Lambda$. We denote by
\begin{equation}
\begin{aligned}
\mathcal E_{n,\lambda,\Lambda}\coloneqq\bigl\{G\in C(\Sn):\;&G(0)=0,\\[-2mm]
&\lambda\operatorname{tr}N
 \le G(M+N)-G(M)
 \le\Lambda\operatorname{tr}N\\[-1mm]
&\hspace{42mm}\text{for all }M\in\Sn,\ N\ge0\bigr\}.
\end{aligned}
\label{eq:intro-ellipticity}
\end{equation}
the class of normalized uniformly elliptic operators. We equip it with
locally uniform convergence; with this topology it is compact. If
$G\in\mathcal E_{n,\lambda,\Lambda}$ and $t>0$, set
\begin{equation}
 (T_tG)(M)\coloneqq tG(t^{-1}M).
 \label{eq:intro-scaling-action}
\end{equation}
This action records exactly what happens under normalization. Positive one-homogeneous operators are precisely the fixed points of this action.

This action is present in several asymptotic regularity theories.
The recession regime $t\downarrow0$ appears in
\cite{SilvestreTeixeira2015,PimentelTeixeira2016}. Savin's small-perturbation
theorem \cite{Savin2007} as well as \cite{DosPrazeresTeixeira2016} work at the opposite endpoint
$t\uparrow\infty$.
The action also arises naturally in singularly perturbed free boundary
models from combustion theory,
$F_\varepsilon=T_\varepsilon F$, see \cite{RicarteTeixeira2011} and also \cite{NST}. 

Those theories exploit a favorable limiting operator. Our viewpoint is different. We retain the entire orbit and ask only for qualitative smoothness. For a single operator $F$, define its \emph{compact scaling hull} by
\begin{equation}
 \mathcal H(F)\coloneqq
 \overline{\{T_tF:t>0\}}^{\,C_{\mathrm{loc}}(\Sn)}.
 \label{eq:intro-scaling-hull}
\end{equation}
Intermediate scales cannot in general be discarded: an iteration may move
through the whole hull before it sees either endpoint. The gain is that no
estimate is assumed for any member of the hull. Only qualitative smoothness
is required, one equation at a time.

The first theorem makes this precise. Let
$\mathfrak F\subset\mathcal E_{n,\lambda,\Lambda}$ be compact and satisfy
$T_t\mathfrak F=\mathfrak F$ for every $t>0$. If every local solution of
every equation in $\mathfrak F$ belongs to $C^{1,\gamma}_{\mathrm{loc}}$,
then, for every $0<\beta<\gamma$,
\begin{equation}
 R\|Du\|_{L^\infty(B_{R/2}(x_0))}
 +R^{1+\beta}[Du]_{C^{0,\beta}(B_{R/2}(x_0))}
 \le C_{\beta,\mathfrak F}\osc_{B_R(x_0)}u.
 \label{eq:family-estimate}
\end{equation}
No $C^{1,\gamma}$ norm occurs in the hypothesis. Heuristically, the mechanism is simple.
For a normalized profile $w$, qualitative $C^{1,\gamma}$ regularity gives
\[
 \|w-\ell_w\|_{L^\infty(B_r)}
 \le L_wr^{1+\gamma}
 =r^{1+\beta}\bigl(L_wr^{\gamma-\beta}\bigr).
\]
The constant $L_w$ may depend arbitrarily on $w$, but the factor in
parentheses is small at some radius. Compactness turns these profile-dependent
radii into a finite range of choices, and scaling invariance permits
iteration. This also identifies the critical threshold exactly, see Corollary \ref{cor:critical-exponents}. Thus the exponent seen at the level of individual solutions is the same as
the exponent controlled uniformly by the equation.

At second order, subtracting a quadratic polynomial translates the operator
in matrix space. The relevant action is
\begin{equation}
 (\mathcal S_{A,t}G)(M)
 \coloneqq t\bigl[G(A+t^{-1}M)-G(A)\bigr].
 \label{eq:intro-centered-action}
\end{equation}
If a compact family is closed under this action along the zero level
$G(A)=0$, and all of its solutions are twice differentiable at every point,
then the family admits a uniform $C^{2,\alpha_*}$ estimate for some
$\alpha_*>0$. No modulus for the quadratic remainder is assumed. If the
qualitative information is $C^{2,\gamma}_{\mathrm{loc}}$, one obtains every
exponent $\beta<\gamma$. In particular, everywhere twice differentiability
for the full class $\mathcal E_{n,\lambda,\Lambda}$ is equivalent to a
universal $C^{2,\alpha}$ theory.

For a single operator, less information is enough. At a point where $u$ is
twice differentiable, let $P$ be its quadratic Taylor polynomial and set
$a(r)=\|u-P\|_{L^\infty(B_r(x_0))}=o(r^2)$. Normalizing this error produces the
parameter $r^2/a(r)\to\infty$. This singles out the one-sided centered tangent
hull
\begin{equation}
 \mathfrak T^{(2)}_\infty(F)
 \coloneqq\left\{H:\ \mathcal S_{A_j,t_j}F\longrightarrow H,\quad
 F(A_j)=0,\quad t_j\longrightarrow\infty\right\}.
 \label{eq:intro-tangent-hull}
\end{equation}
If solutions of these tangent equations are twice differentiable everywhere,
then every point of twice differentiability of an $F$-solution is a regular
point. More precisely, for some $\alpha_{\mathrm{tan}}>0$,
\begin{equation}
 \Reg(u)=\{x:u\text{ is twice differentiable at }x\}
 =\{x:u\in C^2\text{ in a neighborhood of }x\},
 \label{eq:intro-regular-set}
\end{equation}
and $u\in C^{2,\beta}_{\mathrm{loc}}(\Reg(u))$ for every
$\beta<\alpha_{\mathrm{tan}}$. The singular set is relatively closed and has
locally finite $\mathcal H^{n-\varepsilon}$ measure for a universal
$\varepsilon>0$. Here compactness first produces a uniform flatness threshold;
the $W^{3,\varepsilon}$ distribution estimate then controls the exceptional
set \cite[Lemma~5.2]{ArmstrongSilvestreSmart2012}. The formulation is intrinsic
to the zero equation. In particular, it covers tangent operators whose zero
equations are linear without requiring any rate of convergence to them.

The present work is devoted to the homogeneous equation $F=0$. As is customary in the regularity theory of nonlinear PDEs, once interior a priori regularity is established for this prototype setting, corresponding results for nonhomogeneous equations can be derived through compactness arguments of the type introduced in \cite{Caff89}; see also \cite{Teix14}.

The last section gives a concrete use of this criterion. We consider
operators with finitely many switching levels and show that their centered
tangents are linear or have a single convex or concave hinge. We then
construct a genuinely nonsmooth Isaacs operator in dimension three. It is
neither convex nor concave and has no $C^1$ uniformly elliptic representative
with the same zero equation, yet all of its centered tangent equations are
convex or concave. The resulting partial regularity theorem is therefore not
available from the usual differentiability or convexity assumptions.

The paper is organized as follows. Sections~\ref{sec:framework}--\ref{sec:proof-and-consequences} contain the
compactness and affine-improvement argument. Centered renormalization is
developed in Sections~\ref{sec:second-order-self-improvement} and
\ref{sec:partial-regularity}; the switching application is given in
Section~\ref{sec:switching-equations}.

\section{Mathematical framework}\label{sec:framework}

We write \(B_r(x_0)\) for the open ball of radius \(r\) centered at \(x_0\), and \(B_r\coloneqq B_r(0)\).  For a bounded function \(u\) on a set \(E\), set
\[
  \osc_Eu
  \coloneqq
  \sup_Eu-\inf_Eu.
\]
For \(0<\beta\leq1\), set
\[
  [g]_{C^{0,\beta}(E)}
  \coloneqq
  \sup_{\substack{x,y\in E\\x\neq y}}
  \frac{\abs{g(x)-g(y)}}{\abs{x-y}^{\beta}}.
\]
When \(\beta=1\), this is the Lipschitz seminorm. All solutions are understood in the viscosity sense.
We use the Frobenius norm
\[
 |M|\coloneqq\bigl(\operatorname{tr}(M^2)\bigr)^{1/2},
 \qquad M\in\Sn.
\]
For an operator \(G\) and an open set \(\Omega\), we denote by
\(\mathcal H_G(\Omega)\) the set of all $C$-viscosity solutions of
\(G(D^2u)=0\) in \(\Omega\).

\subsection{The operator space and its scaling action}

\begin{lemma}\label{lem:operator-compactness}
The space \(\mathcal E_{n,\lambda,\Lambda}\) is compact in the topology of locally uniform convergence.  Moreover, \(T_t\) is a homeomorphism of \(\mathcal E_{n,\lambda,\Lambda}\) with inverse \(T_{1/t}\), and
\[
  T_s(T_tG)=T_{st}G
\]
for every \(G\in\mathcal E_{n,\lambda,\Lambda}\) and \(s,t>0\).
\end{lemma}

\begin{proof}
The ellipticity inequalities and the normalization \(G(0)=0\) imply
\[
  \abs{G(M)-G(N)}
  \leq
  C(n,\lambda,\Lambda)\abs{M-N} \quad \text{and} \quad \abs{G(M)}
  \leq
  C(n,\lambda,\Lambda)\abs{M}
\]
for every \(G\in\mathcal E_{n,\lambda,\Lambda}\).  Indeed, decompose \(M-N\) into its positive and negative parts and apply \eqref{eq:intro-ellipticity} twice. Arzel\`a--Ascoli on bounded subsets of the finite-dimensional space \(\Sn\), followed by a diagonal argument, gives sequential compactness. The locally uniform topology is
metrizable, and both normalization and ellipticity pass to locally uniform limits.  Hence \(\mathcal E_{n,\lambda,\Lambda}\) is compact.

For \(P\geq0\),
\[
  (T_tG)(M+P)-(T_tG)(M)
  =
  t\left[
    G(t^{-1}M+t^{-1}P)-G(t^{-1}M)
  \right],
\]
so \(T_tG\) has the same ellipticity constants and is normalized at the origin. The identities \(T_sT_t=T_{st}\) and \((T_t)^{-1}=T_{1/t}\) follow directly from \eqref{eq:intro-scaling-action}. Local uniform continuity of \(T_t\) and its inverse is immediate.
\end{proof}

\begin{lemma}\label{lem:solution-rescaling}
Let \(a,r>0\), \(G\in\mathcal E_{n,\lambda,\Lambda}\), and
\(u\in\mathcal H_G(B_r(x_0))\). If \(\ell\) is affine, then
\[
  v(x)
  \coloneqq
  \frac{u(x_0+rx)-\ell(x_0+rx)}{a} \quad \text{solves} \quad (T_{r^2/a}G)(D^2v)=0
\]
wherever it is defined.
\end{lemma}

\begin{proof}
Since affine functions have zero Hessian,
\[
  D^2v(x)=\frac{r^2}{a}D^2u(x_0+rx).
\]
The asserted equation follows from the definition of \(T_{r^2/a}\).
The same calculation with touching test functions proves the statement in
the viscosity sense.
\end{proof}

The scaling hull $\mathcal H(F)$ defined in
\eqref{eq:intro-scaling-hull} is compact by
\Cref{lem:operator-compactness}; the identity
$T_sT_t=T_{st}$ and the continuity of $T_s$ give
\[
  T_s\mathcal H(F)=\mathcal H(F)
  \qquad(s>0).
\]

\subsection{Compactness of normalized pairs}

Let \(\mathfrak F\subset\mathcal E_{n,\lambda,\Lambda}\) be nonempty and compact. Define
\[
  \mathcal P_0(\mathfrak F)
  \coloneqq
  \left\{
    \left(G,v|_{\overline B_{1/2}}\right)
    \colon
    \begin{array}{l}
      G\in\mathfrak F,\quad
      v\in\mathcal H_G(B_1),\\
      \norm{v}_{L^\infty(B_1)}\leq1
    \end{array}
  \right\},
\]
and let
\[
  \mathcal P(\mathfrak F)
  \coloneqq
  \overline{\mathcal P_0(\mathfrak F)}^{
    \,C_{\mathrm{loc}}(\Sn)\times C(\overline B_{1/2})
  }.
\]

\begin{lemma}\label{lem:pair-compactness}
The set \(\mathcal P(\mathfrak F)\) is compact. If \((G,w)\in\mathcal P(\mathfrak F)\), then \(G\in\mathfrak F\) and
\[
  G(D^2w)=0
  \qquad\text{in }B_{1/2}.
\]
\end{lemma}

\begin{proof}
Applying the Krylov--Safonov interior H\"older estimate first on a slightly
larger ball, say \(B_{3/4}\), gives
\(\alpha=\alpha(n,\lambda,\Lambda)\in(0,1)\) and
\(C=C(n,\lambda,\Lambda)\) such that
\[
  [v]_{C^{0,\alpha}(\overline B_{1/2})}
  \leq C
\]
for every pair in \(\mathcal P_0(\mathfrak F)\). The first components range in the compact family \(\mathfrak F\), and the second components are uniformly bounded and equicontinuous. Arzel\`a--Ascoli therefore gives compactness of the closure.

If $G_j\to G$ locally uniformly on $\Sn$ and $v_j\to w$ uniformly on
$\overline B_{1/2}$,
then \(G\in\mathfrak F\), and stability of viscosity solutions for locally uniformly convergent operators yields
\[
  G(D^2w)=0
  \qquad\text{in }B_{1/2}.
\]
\end{proof}

\section{Compact affine improvement and iteration}
\label{sec:finite-menu}

This is where the qualitative hypothesis enters the proof.  The scale at
which one particular solution becomes flat may be very small.  Compactness
says that, among all normalized profiles, these scales remain in a controlled
range.

\begin{lemma}\label{lem:finite-menu}
Let \(0<\beta<\gamma\leq1\), and let \(\mathfrak F\subset\mathcal E_{n,\lambda,\Lambda}\) be nonempty and compact in the topology of locally uniform convergence. Assume that, for every \(G\in\mathfrak F\), every ball \(B\subset\Rn\), and every
\(u\in\mathcal H_G(B)\),
\[
  u\in C^{1,\gamma}_{\mathrm{loc}}(B).
\]

Then there is \(\theta\in(0,1/4)\) such that, for
every \(G\in\mathfrak F\) and every \(v\in\mathcal H_G(B_1)\) satisfying
\(\norm{v}_{L^\infty(B_1)}\leq1\), one can find
\(r\in[\theta,1/4]\) and an affine function
\(\ell(x)=a+b\cdot x\) for which
\begin{equation}
  \norm{v-\ell}_{L^\infty(B_r)}
  \leq\frac12r^{1+\beta}.
  \label{eq:finite-menu-improvement}
\end{equation}
\end{lemma}

\begin{proof}
Fix \(q=(G,w)\in\mathcal P(\mathfrak F)\). By
\Cref{lem:pair-compactness}, \(w\) solves the equation for \(G\) in
\(B_{1/2}\), and the qualitative hypothesis allows us to set
\[
  \ell_q(x)
  \coloneqq
  w(0)+Dw(0)\cdot x,
\]
and hence there exists \(L_q<\infty\) such that
\[
  \norm{w-\ell_q}_{C(\overline B_r)}
  \leq
  L_qr^{1+\gamma}
  \qquad
  \left(0<r\leq\frac14\right).
\]
Because \(\gamma-\beta>0\), we may choose \(\rho_q\in(0,1/4)\) sufficiently small that
\[
  L_q\rho_q^{\gamma-\beta}<\frac14.
\]
It follows that
\begin{equation}
  \norm{w-\ell_q}_{C(\overline B_{\rho_q})}
  <
  \frac14\rho_q^{1+\beta} \eqqcolon \varepsilon_q
  \label{eq:strict-flatness}
\end{equation}
Set
\[
  \mathcal O_q
  \coloneqq
  \left\{
    (H,z)\in\mathcal P(\mathfrak F)\colon
    \norm{z-w}_{C(\overline B_{1/2})}
    <
    \varepsilon_q
  \right\}.
\]
This is a relative neighborhood of \(q\) in \(\mathcal P(\mathfrak F)\), and every
profile in it satisfies the desired inequality at the same radius and with
the same affine function.  Compactness gives a finite subcover with centers
\[
  q_i=(G_i,w_i)\in\mathcal P(\mathfrak F), \qquad 1\leq i\leq N.
\]
Set
\[
  \rho_i\coloneqq\rho_{q_i},
  \qquad
  \ell_i\coloneqq\ell_{q_i}.
\]
Every admissible pair belongs to some \(\mathcal O_{q_i}\), and hence
\[
  \begin{aligned}
  \norm{v-\ell_i}_{C(\overline B_{\rho_i})}
  &\leq
  \norm{v-w_i}_{C(\overline B_{\rho_i})}
  +
  \norm{w_i-\ell_i}_{C(\overline B_{\rho_i})}
  \\
  &\leq
  \norm{v-w_i}_{C(\overline B_{1/2})}
  +
  \norm{w_i-\ell_i}_{C(\overline B_{\rho_i})}
  \\
  &<
  \frac14\rho_i^{1+\beta}
  +
  \frac14\rho_i^{1+\beta}
  =
  \frac12\rho_i^{1+\beta}.
  \end{aligned}
\]
Set \(\theta\coloneqq\min_i\rho_i\). This proves the lemma.
\end{proof}

\begin{lemma}\label{lem:iteration}
Assume the hypotheses of \Cref{lem:finite-menu}. Let
\(G\in\mathfrak F\) and \(v\in\mathcal H_G(B_1)\) satisfy
\(\norm{v}_{L^\infty(B_1)}\leq1\). Assume, in addition, that
\[
  T_t\mathfrak F=\mathfrak F
  \qquad\text{for every }t>0.
\]
There exist
\[
  r_k\in[\theta,1/4],
  \qquad
  R_0\coloneqq1,
  \qquad
  R_{k+1}\coloneqq R_k r_k,
\]
operators \(G_k\in\mathfrak F\), and affine functions \(P_k\), with \(G_0\coloneqq G\) and \(P_0\coloneqq0\), such that
\begin{equation}
  \norm{v-P_k}_{L^\infty(B_{R_k})}
  \leq
  \frac12R_k^{1+\beta}
  \qquad(k\geq1).
  \label{eq:iterated-flatness}
\end{equation}
Moreover,
\begin{equation}
  v_k(x)
  \coloneqq
  \frac{v(R_kx)-P_k(R_kx)}{R_k^{1+\beta}}
  \label{eq:normalized-iterate}
\end{equation}
belongs to \(\mathcal H_{G_k}(B_1)\), \(\norm{v_k}_{L^\infty(B_1)}\leq1\), with \(G_k=T_{R_k^{1-\beta}}G\).
\end{lemma}

\begin{proof}
Set \(v_0\coloneqq v\). Suppose \(v_k\) and \(G_k\) have been constructed.
\Cref{lem:finite-menu} provides \(r_k\in[\theta,1/4]\) and an affine
function \(\widehat\ell_k\) such that
\begin{equation}
  \norm{v_k-\widehat\ell_k}_{L^\infty(B_{r_k})}
  \leq
  \frac12r_k^{1+\beta}.
  \label{eq:one-step-flatness}
\end{equation}
Define
\[
\begin{aligned}
 v_{k+1}(x)
 &\coloneqq\frac{v_k(r_kx)-\widehat\ell_k(r_kx)}{r_k^{1+\beta}},\\
 P_{k+1}(y)
 &\coloneqq P_k(y)+R_k^{1+\beta}\widehat\ell_k(y/R_k),\\
 G_{k+1}&\coloneqq T_{r_k^{1-\beta}}G_k.
\end{aligned}
\]
Then \(\norm{v_{k+1}}_{L^\infty(B_1)}\leq1/2\). By
\Cref{lem:solution-rescaling}, \(v_{k+1}\) solves the equation for
\(G_{k+1}\), and the invariance hypothesis of the lemma gives
\(G_{k+1}\in\mathfrak F\). The semigroup identity gives
\[
  G_{k+1}
  =
  T_{R_{k+1}^{1-\beta}}G.
\]
Substituting \eqref{eq:normalized-iterate} into the definition of \(v_{k+1}\) shows that the same formula holds at index \(k+1\). Finally, rescaling \eqref{eq:one-step-flatness} back to \(B_{R_{k+1}}\) gives \eqref{eq:iterated-flatness}.
\end{proof}

\begin{proposition}\label{prop:pointwise-expansion}
Assume the setting of \Cref{lem:iteration}. There is
$A_{\beta,\mathfrak F}<\infty$ such that, for every $G\in\mathfrak F$ and
every $v\in\mathcal H_G(B_1)$ with
$\|v\|_{L^\infty(B_1)}\leq1$, the function $v$ is differentiable at the
origin and satisfies
\begin{equation}
  \abs{v(x)-v(0)-Dv(0)\cdot x}
  \leq
  A_{\beta,\mathfrak F}\abs{x}^{1+\beta}
  \qquad(x\in B_\theta).
  \label{eq:pointwise-expansion}
\end{equation}
Moreover,
\[
  \abs{Dv(0)}
  \leq
  A_{\beta,\mathfrak F}.
\]
\end{proposition}

\begin{proof}
With the notation of \Cref{lem:iteration}, write
\[
  \widehat\ell_k(x)=a_k+b_k\cdot x,
  \qquad
  P_k(x)=c_k+p_k\cdot x.
\]
Since \(\norm{v_k}_{L^\infty(B_1)}\leq1\),
\eqref{eq:one-step-flatness} implies
\[
  \norm{\widehat\ell_k}_{L^\infty(B_{r_k})}
  \leq\frac32,
  \qquad
  \abs{b_k}
  \leq
  \frac{3}{2r_k}
  \leq
  \frac{3}{2\theta}.
\]
The recursion from \Cref{lem:iteration} gives
\[
  p_{k+1}-p_k=R_k^\beta b_k.
\]
Since \(R_{k+j}\leq 4^{-j}R_k\), the sequence \((p_k)_{k \in \mathbb{N}}\) converges to
some \(p\in\Rn\), and
\[
  \abs{p-p_k}
  \leq
  C R_k^\beta.
\]
Moreover, evaluating \eqref{eq:iterated-flatness} at the origin yields
\[
  \abs{v(0)-c_k}
  \leq
  \frac12R_k^{1+\beta}.
\]
Given \(0<\abs{x}<\theta\), choose \(k\geq1\) such that
\[
  R_{k+1}\leq\abs{x}<R_k.
\]
Since \(R_k\le4^{-k}\), these radii tend to zero and such a \(k\) exists.
Using the preceding estimates and \eqref{eq:iterated-flatness}, we find
\[
\begin{aligned}
  \abs{v(x)-v(0)-p\cdot x}
  &\leq
  \|v-P_k\|_{L^\infty(B_{R_k})}
  +\abs{c_k-v(0)}+\abs{p_k-p}\,\abs{x}\\
  &\leq C R_k^{1+\beta}.
\end{aligned}
\]
Because \(R_{k+1}\geq\theta R_k\), it follows that
\[
  \abs{v(x)-v(0)-p\cdot x}
  \leq
  C\theta^{-(1+\beta)}\abs{x}^{1+\beta}.
\]
Thus \(v\) is differentiable at the origin, \(Dv(0)=p\), and
\eqref{eq:pointwise-expansion} follows.

Finally, evaluating \eqref{eq:pointwise-expansion} at
\(x=(\theta/2)e\), where \(e\) is any unit vector, and using
\(\norm{v}_{L^\infty(B_1)}\leq1\), gives a uniform bound for
\(\abs{Dv(0)}\).  Increasing \(A_{\beta,\mathfrak F}\), if necessary,
completes the proof.
\end{proof}

\section{Uniform gradient estimates}\label{sec:proof-and-consequences}

\begin{theorem}[Qualitative regularity gives uniform estimates]
\label{thm:intro-main}
Let \(0<\gamma\le1\), and let
\(\mathfrak F\subset\mathcal E_{n,\lambda,\Lambda}\) be nonempty and
compact.  Assume
\begin{equation}
 T_t\mathfrak F=\mathfrak F\qquad(t>0)
 \label{eq:family-invariance}
\end{equation}
and suppose every local solution of every equation in \(\mathfrak F\)
belongs to \(C^{1,\gamma}_{\mathrm{loc}}\).  Then, for every
\(0<\beta<\gamma\), there is \(C_{\beta,\mathfrak F}<\infty\) such that
\eqref{eq:family-estimate} holds for every \(G\in\mathfrak F\) and every
\(u\in\mathcal H_G(B_R(x_0))\).
\end{theorem}

\begin{proof}
Translation and the change of variables $x=x_0+Ry$ reduce the problem to
$B_1$; the new operator is $T_{R^2}G$, which still belongs to
$\mathfrak F$. It remains to normalize the oscillation.
 
Let \(G\in\mathfrak F\), \(u\in\mathcal H_G(B_1)\), and
\(M\coloneqq\osc_{B_1}u\). If \(M=0\), then \(u\) is constant and there is
nothing to prove. Assume \(M>0\), fix \(x\in B_{1/2}\), and define
\[
  w_x(y)
  \coloneqq
  \frac{u(x+4^{-1}y)-u(x)}{M},
  \qquad y\in B_1.
\]
Since \(x+\frac14B_1\subset B_{3/4}\subset B_1\), the function \(w_x\) is
defined on \(B_1\), and \(\norm{w_x}_{L^\infty(B_1)}\leq1\). By
\Cref{lem:solution-rescaling}, \(w_x\) solves the equation for
\(T_{4^{-2}/M}G\), which belongs to $\mathfrak F$ by
\eqref{eq:family-invariance}. Applying \Cref{prop:pointwise-expansion} to
\(w_x\) and returning to the original variables gives
\begin{equation}
  \abs{u(x+h)-u(x)-Du(x)\cdot h}
  \leq
  KM\abs{h}^{1+\beta}
  \qquad(\abs{h}\leq r_0),
  \label{eq:center-expansion}
\end{equation}
where
\[
  r_0\coloneqq 4^{-1}\theta,
  \qquad
  K\coloneqq A_{\beta,\mathfrak F}4^{1+\beta}.
\]
The gradient bound in \Cref{prop:pointwise-expansion} also gives
\begin{equation}
  \abs{Du(x)}
  \leq
  4A_{\beta,\mathfrak F}M
  \qquad(x\in B_{1/2}).
  \label{eq:gradient-bound}
\end{equation}

Let \(x,y\in B_{1/2}\), set \(d\coloneqq\abs{x-y}\), and first suppose \(0<d\leq r_0\).  Define
\[
  m\coloneqq\frac{x+y}{2},
  \qquad
  \ell_x(z)
  \coloneqq
  u(x)+Du(x)\cdot(z-x),
  \qquad
  \ell_y(z)
  \coloneqq
  u(y)+Du(y)\cdot(z-y).
\]
For \(z\in B_{d/2}(m)\), both \(\abs{z-x}\) and \(\abs{z-y}\) are at most \(d\). Hence \eqref{eq:center-expansion} gives
\[
  \norm{\ell_x-\ell_y}_{L^\infty(B_{d/2}(m))}
  \leq
  2KMd^{1+\beta}.
\]
If \(L\) is affine on \(B_r(m)\), then
\[
  \abs{DL}
  \leq
  r^{-1}\norm{L}_{L^\infty(B_r(m))}.
\]
Applying this to \(L=\ell_x-\ell_y\) with \(r=d/2\) yields
\[
  \abs{Du(x)-Du(y)}
  \leq
  4KMd^\beta.
\]
If \(d>r_0\), then \eqref{eq:gradient-bound} gives
\[
  \abs{Du(x)-Du(y)}
  \leq
  8A_{\beta,\mathfrak F}
  r_0^{-\beta}Md^\beta.
\]
Together with \eqref{eq:gradient-bound}, this proves the normalized form of
\eqref{eq:family-estimate}.
\end{proof}

\begin{remark}[The endpoint]
The strict inequality \(\beta<\gamma\) is used only to create a small factor
at a profile-dependent radius.  The same proof reaches \(\beta=\gamma\) if,
for every closure pair \((G,w)\in\mathcal P(\mathfrak F)\), the profile \(w\)
satisfies
\[
 \liminf_{r\downarrow0}
 \frac{\|w-w(0)-Dw(0)\cdot(\,\cdot\,)\|_{L^\infty(B_r)}}
 {r^{1+\gamma}}=0.
\]
No endpoint assertion is made without such additional information.
\end{remark}

\begin{corollary}\label{cor:scaling-hull}
Let \(F\in\mathcal E_{n,\lambda,\Lambda}\) and \(0<\gamma\leq1\). Suppose every local solution of every equation
\[
  G(D^2u)=0,
  \qquad G\in\mathcal H(F),
\]
belongs to \(C^{1,\gamma}_{\mathrm{loc}}\). Then, for every
\(0<\beta<\gamma\), solutions of \(F(D^2u)=0\) satisfy \eqref{eq:family-estimate}, with a constant depending on \(\mathcal H(F)\).
\end{corollary}

\begin{proof}
The compactness and invariance of \(\mathcal H(F)\) were established in
\Cref{sec:framework}. Apply \Cref{thm:intro-main} to this family
and then take \(G=F\).
\end{proof}

In particular, if $F\in\mathcal E_{n,\lambda,\Lambda}$ is positively
one-homogeneous, qualitative
$C^{1,\gamma}_{\mathrm{loc}}$ smoothness for its solutions implies
\eqref{eq:family-estimate} for every $\beta<\gamma$. The same conclusion
holds uniformly for any compact family of positively one-homogeneous
operators with common ellipticity constants, provided every local solution
for every member of the family has the stated qualitative regularity.

The qualitative and quantitative regularity thresholds therefore coincide.

\begin{corollary}\label{cor:critical-exponents}
Let \(\mathfrak F\subset\mathcal E_{n,\lambda,\Lambda}\) be a nonempty compact scaling-invariant family. Define
\begin{align*}
  \Qual(\mathfrak F)
  &\coloneqq
  \sup\Bigl\{
    \alpha\in(0,1]
    \colon
    \mathcal H_G(B)
    \subset
    C^{1,\alpha}_{\mathrm{loc}}(B)
    \text{ for every \(G\in\mathfrak F\) and every ball \(B\)}
  \Bigr\},\\
  \Est(\mathfrak F)
  &\coloneqq
  \sup\Bigl\{
    \alpha\in(0,1]
    \colon
    \eqref{eq:family-estimate}
    \text{ holds with \(\beta=\alpha\), uniformly over \(\mathfrak F\)}
  \Bigr\}.
\end{align*}
Then both defining sets are nonempty and
\[
  \Qual(\mathfrak F) = \Est(\mathfrak F).
\]
\end{corollary}

\begin{proof}
The universal interior $C^{1,\alpha_0}$ estimate
\cite[Corollary~5.7]{CaffarelliCabre1995} supplies an exponent
\[
  \alpha_0=\alpha_0(n,\lambda,\Lambda)>0,
\]
so both defining sets are nonempty. Every uniform estimate implies the corresponding qualitative
regularity, and therefore
\[
  \Est(\mathfrak F)\leq\Qual(\mathfrak F).
\]
Conversely, fix \(0<\beta<\Qual(\mathfrak F)\). By the definition of the
supremum, there is an exponent \(\gamma>\beta\) for which every local solution
of every equation in \(\mathfrak F\) belongs to
\(C^{1,\gamma}_{\mathrm{loc}}\). Applying \Cref{thm:intro-main} gives the
uniform estimate at exponent \(\beta\), hence
\(\beta\leq\Est(\mathfrak F)\). Letting
\(\beta\uparrow\Qual(\mathfrak F)\) proves the reverse inequality.
\end{proof}

\section{Centered renormalization and Hessian estimates}
\label{sec:second-order-self-improvement}

At the Hessian level the compactness mechanism yields more than uniformity of
a prescribed exponent: the existence of pointwise quadratic expansions alone
forces a positive H\"older exponent. The necessary renormalization now has two
parameters, reflecting both amplitude scaling and subtraction of a tangent
quadratic.

\subsection{Centered quadratic renormalization}

Subtracting an affine function leaves a translation-invariant equation unchanged. Subtracting a quadratic does not: it translates the operator in matrix space. For \(G\in\mathcal E_{n,\lambda,\Lambda}\), \(A\in\Sn\), and \(t>0\), define the centered quadratic renormalization
\[
  (\mathcal S_{A,t}G)(M)
  \coloneqq
  t\left[G(A+t^{-1}M)-G(A)\right].
\]
It belongs to \(\mathcal E_{n,\lambda,\Lambda}\), and direct calculation gives
\begin{equation}
  \mathcal S_{B,s}\bigl(\mathcal S_{A,t}G\bigr)
  =
  \mathcal S_{A+t^{-1}B,st}G,
  \qquad
  \mathcal S_{0,t}G=T_tG.
  \label{eq:centered-composition}
\end{equation}
We shall also use the joint continuity of this action: if $H_j\to H$
locally uniformly, $B_j\to B$, and $s_j\to s>0$, then
$\mathcal S_{B_j,s_j}H_j\to\mathcal S_{B,s}H$ locally uniformly. This
follows from the common Lipschitz bound for uniformly elliptic operators and
compactness of the matrix ranges involved.
If \(G(A)=0\), \(\pi\) is a quadratic polynomial with
\(D^2\pi=A\), and
\[
  v(x)
  \coloneqq
  \frac{u(x_0+rx)-\pi(x_0+rx)}{a},
  \qquad a,r>0,
\]
then \(G(D^2u)=0\) implies
\[
  \mathcal S_{A,r^2/a}G(D^2v)=0.
\]
Thus \(\mathcal S_{A,t}\), rather than \(T_t\) alone, is the action
closed under quadratic improvement of flatness.

\begin{lemma}\label{lem:zero-level-correction}
For every \(H\in\mathcal E_{n,\lambda,\Lambda}\) and \(A\in\Sn\),
there is a unique \(\tau_A(H)\in\R\) such that
\[
  H\bigl(A+\tau_A(H)I\bigr)=0.
\]
Moreover,
\begin{equation}
  \abs{\tau_A(H)}
  \leq
  \frac{\abs{H(A)}}{n\lambda}.
  \label{eq:zero-level-correction-bound}
\end{equation}
In particular, if \(H_j\to H\) locally uniformly and \(H(A)=0\), then
\(\tau_A(H_j)\to0\).
\end{lemma}

\begin{proof}
For $s_2>s_1$, uniform ellipticity gives
\[
 n\lambda(s_2-s_1)
 \leq H(A+s_2I)-H(A+s_1I)
 \leq n\Lambda(s_2-s_1).
\]
Thus $s\mapsto H(A+sI)$ is strictly increasing and onto $\R$, proving
existence and uniqueness. Comparison with $s=0$ gives
\eqref{eq:zero-level-correction-bound}; the final assertion follows from
local uniform convergence evaluated at $A$.
\end{proof}

\begin{definition}
For \(F\in\mathcal E_{n,\lambda,\Lambda}\), define
\[
 \mathcal H^{(2)}(F)\coloneqq
 \overline{\{\mathcal S_{A,t}F:F(A)=0,\ t>0\}}^{
 C_{\mathrm{loc}}(\Sn)}.
\]
\end{definition}

\begin{proposition}\label{prop:centered-hull}
The family \(\mathcal H^{(2)}(F)\) is compact, contains \(F\), and is closed
under every admissible centered renormalization: if
\(H\in\mathcal H^{(2)}(F)\), \(H(B)=0\), and \(s>0\), then
\(\mathcal S_{B,s}H\in\mathcal H^{(2)}(F)\).
\end{proposition}

\begin{proof}
Only the last assertion needs proof.  Choose
\[
 H_j=\mathcal S_{A_j,t_j}F\longrightarrow H,
 \qquad F(A_j)=0.
\]
Set $B_j=B+\tau_B(H_j)I$. The zero-level correction gives $B_j\to B$
and $H_j(B_j)=0$.
By \eqref{eq:centered-composition},
\[
 \mathcal S_{B_j,s}H_j
 =\mathcal S_{A_j+t_j^{-1}B_j,st_j}F.
\]
The new center belongs to \(F^{-1}(0)\), because \(H_j(B_j)=0\).
Joint continuity of the centered action permits passage to the limit and
proves the assertion. Compactness follows from
\Cref{lem:operator-compactness}, and
\(F=\mathcal S_{0,1}F\).
\end{proof}

A function is \emph{twice differentiable at $x_0$} if, for some
$p\in\Rn$ and $A\in\Sn$,
\begin{equation}
  u(x_0+h)=u(x_0)+p\cdot h+\frac12h\cdot Ah+o(|h|^2).
  \label{eq:quadratic-expansion}
\end{equation}
This pointwise property carries no modulus and gives no continuity, or even
local boundedness, of the resulting Hessian.

\begin{lemma}
\label{lem:quadratic-compatibility}
Let $G\in\mathcal E_{n,\lambda,\Lambda}$ and let $u$ be a viscosity
solution of $G(D^2u)=0$ near $x_0$.  If $u$ has the expansion
\eqref{eq:quadratic-expansion} at \(x_0\), then \(G(A)=0\).
\end{lemma}

\begin{proof}
Let $P(x_0+h)\coloneqq u(x_0)+p\cdot h+\frac12h\cdot Ah$.  For every
$\varepsilon>0$, choose a ball on which
$|u-P|\le\varepsilon|x-x_0|^2$.  Then
$P+2\varepsilon|x-x_0|^2$ touches $u$ from above at $x_0$ and
$P-2\varepsilon|x-x_0|^2$ touches it from below there.  Under the
increasing-operator convention, the viscosity inequalities give
$G(A-4\varepsilon I)\le0\le G(A+4\varepsilon I)$.  Continuity of $G$ and
$\varepsilon\downarrow0$ give $G(A)=0$.
\end{proof}

\begin{lemma}
\label{lem:quadratic-coefficients}
There is $C_n<\infty$ such that every quadratic polynomial $Q$ satisfies
\begin{equation}
 |D^2Q|+r^{-1}|DQ(x_0)|+r^{-2}|Q(x_0)|
 \le C_nr^{-2}\|Q\|_{L^\infty(B_r(x_0))}.
 \label{eq:quadratic-coefficient-estimate}
\end{equation}
\end{lemma}

\begin{proof}
After translation and scaling it is enough to take $x_0=0$ and $r=1$.
On the finite-dimensional space of quadratic polynomials, the coefficient
norm on the left and the $L^\infty(B_1)$ norm are equivalent.
\end{proof}

\begin{theorem}
\label{thm:qualitative-c2-self-improvement}
Let $\mathfrak F\subset\mathcal E_{n,\lambda,\Lambda}$ be a nonempty compact
family satisfying
\begin{equation}
  \mathcal S_{A,t}G\in\mathfrak F
  \qquad\text{whenever }G\in\mathfrak F,\quad G(A)=0,\quad t>0.
  \label{eq:centered-forward-invariance}
\end{equation}
Assume that every $u\in\mathcal H_G(B)$, for every $G\in\mathfrak F$ and
every ball $B\subset\Rn$, is twice differentiable at every point of $B$.
Then there exist $\alpha_*=\alpha_*(\mathfrak F)\in(0,1)$ and
$C_*=C_*(\mathfrak F)<\infty$ such that
\begin{equation}
  R\|Du\|_{L^\infty(B_{R/2}(x_0))}
  +R^2\|D^2u\|_{L^\infty(B_{R/2}(x_0))}
  +R^{2+\alpha_*}[D^2u]_{C^{0,\alpha_*}(B_{R/2}(x_0))}
  \le C_*\osc_{B_R(x_0)}u
  \label{eq:qualitative-c2-estimate}
\end{equation}
for every $G\in\mathfrak F$ and every
$u\in\mathcal H_G(B_R(x_0))$.
\end{theorem}

\begin{proof}
We first extract a uniform quadratic contraction from the pointwise
hypothesis. Fix $\eta\in(0,1/4)$. For
$(G,w)\in\mathcal P(\mathfrak F)$, let
\[
  \pi_w(x)=w(0)+p_w\cdot x+\frac12x\cdot A_wx
\]
be the expansion in \eqref{eq:quadratic-expansion}. By
\Cref{lem:quadratic-compatibility}, $G(A_w)=0$.
Choose $\rho_w\in(0,1/4)$ so that
\begin{equation}
  \|w-\pi_w\|_{L^\infty(B_{\rho_w})}
  <\frac{\eta}{3}\rho_w^2.
  \label{eq:qualitative-quadratic-smallness}
\end{equation}
For a nearby operator $H$, set
\[
  A_{w,H}\coloneqq A_w+\tau_{A_w}(H)I,
  \qquad
  \pi_{w,H}(x)\coloneqq w(0)+p_w\cdot x+\frac12x\cdot A_{w,H}x.
\]
By \Cref{lem:zero-level-correction}, $H(A_{w,H})=0$ and
$A_{w,H}\to A_w$ as $H\to G$. Moreover,
\[
 \|z-\pi_{w,H}\|_{L^\infty(B_{\rho_w})}
 \leq \|z-w\|_{L^\infty(B_{\rho_w})}
 +\|w-\pi_w\|_{L^\infty(B_{\rho_w})}
 +\frac12|\tau_{A_w}(H)|\rho_w^2.
\]
After shrinking a relative neighborhood of $(G,w)$ in
$\mathcal P(\mathfrak F)$, every pair $(H,z)$ in that neighborhood satisfies
\begin{equation}
  \|z-\pi_{w,H}\|_{L^\infty(B_{\rho_w})}
  \le \eta\rho_w^2,
  \qquad H(A_{w,H})=0.
  \label{eq:uniform-quadratic-contraction-local}
\end{equation}

Compactness of \(\mathcal P(\mathfrak F)\) now gives a finite subcover.
Consequently, there are \(K<\infty\) and radii
\[
 \rho_1,\ldots,\rho_N\in(0,1/4)
\]
such that
every normalized pair $(H,z)$ admits, for some
$r\in\{\rho_1,\ldots,\rho_N\}$, a quadratic polynomial
\[
  q(x)=a+b\cdot x+\frac12x\cdot Bx
\]
with
\begin{equation}
  H(B)=0,\qquad
  \|z-q\|_{L^\infty(B_r)}\le\eta r^2,\qquad
  |a|+|b|+|B|\le K.
  \label{eq:finite-quadratic-contraction}
\end{equation}
The last bound follows because only finitely many base jets occur and the
zero-level corrections are small on their chosen neighborhoods.

Set
\[
  \theta\coloneqq\min_i\rho_i,\qquad \Theta\coloneqq\max_i\rho_i.
\]
Take $G\in\mathfrak F$ and a solution $v$ in $B_1$ with
$\|v\|_{L^\infty(B_1)}\le1$. We construct radii $R_k$, polynomials $P_k$,
and normalized solutions $v_k$. Start with $R_0=1$, $P_0=0$, and, writing
$P_k(x)=c_k+p_k\cdot x+\frac12x\cdot A_kx$, define
\begin{equation}
  v_k(x)\coloneqq
  \frac{v(R_kx)-P_k(R_kx)}{\eta^kR_k^2}.
  \label{eq:c2-normalized-iterate}
\end{equation}
Assume $\|v_k\|_{L^\infty(B_1)}\le1$ and that it solves the equation for
\[
  G_k\coloneqq\mathcal S_{A_k,\eta^{-k}}G\in\mathfrak F.
\]
Inductively we also retain
\begin{equation}
 G(A_k)=0,
 \qquad
 \|v-P_k\|_{L^\infty(B_{R_k})}\le\eta^kR_k^2.
 \label{eq:c2-induction-invariants}
\end{equation}
Apply \eqref{eq:finite-quadratic-contraction} to obtain $r_k$ and
$q_k(x)=a_k+b_k\cdot x+\frac12x\cdot B_kx$. Define
\begin{align*}
  R_{k+1}&\coloneqq R_kr_k,\\
  P_{k+1}(y)&\coloneqq P_k(y)+\eta^kR_k^2q_k(y/R_k),\\
  G_{k+1}&\coloneqq\mathcal S_{B_k,\eta^{-1}}G_k.
\end{align*}
Then
\[
  v_{k+1}(x)=
  \frac{v_k(r_kx)-q_k(r_kx)}{\eta r_k^2},
  \qquad \|v_{k+1}\|_{L^\infty(B_1)}\le1.
\]
Moreover, $G_k(B_k)=0$, so centered invariance applies, and the composition
law \eqref{eq:centered-composition} gives
\[
  A_{k+1}=A_k+\eta^kB_k,\qquad
  G_{k+1}=\mathcal S_{A_{k+1},\eta^{-(k+1)}}G.
\]
Since $G_k(B_k)=0$ and $G(A_k)=0$, the definition of $G_k$ also gives
$G(A_{k+1})=0$.  The bound for $v_{k+1}$ is exactly the next error estimate
in \eqref{eq:c2-induction-invariants}.  This closes the induction.

Choose
\begin{equation}
  0<\alpha_*<\min\left\{1,\frac{\log\eta}{\log\theta}\right\}.
  \label{eq:choice-alpha-star}
\end{equation}
Since $R_k\ge\theta^k$, we have $\eta^k\le R_k^{\alpha_*}$.
The polynomial recursion gives
\[
 A_{k+1}-A_k=\eta^kB_k,\qquad
 p_{k+1}-p_k=\eta^kR_kb_k,\qquad
 c_{k+1}-c_k=\eta^kR_k^2a_k.
\]
The coefficient bounds in \eqref{eq:finite-quadratic-contraction} and
$R_{k+j}\le R_k\Theta^j$ show that $P_k$ converges to a quadratic polynomial
$P_\infty(x)=c+p\cdot x+\frac12x\cdot Ax$ and that
\begin{align*}
  |A-A_k|&\le C\eta^k\le CR_k^{\alpha_*},\\
  |p-p_k|&\le C\eta^kR_k\le CR_k^{1+\alpha_*},\\
  |c-c_k|&\le C\eta^kR_k^2\le CR_k^{2+\alpha_*}.
\end{align*}
The error estimate at the origin also gives $c=P_\infty(0)=v(0)$.
Together with \eqref{eq:c2-normalized-iterate}, and then choosing $k$ by
$R_{k+1}\le|x|<R_k$, this yields
\begin{equation}
  |v(x)-v(0)-p\cdot x-\tfrac12x\cdot Ax|
  \le C|x|^{2+\alpha_*}
  \qquad (x\in B_\theta).
  \label{eq:c2-pointwise-expansion}
\end{equation}
In particular $p=Dv(0)$, and $A$ is the unique quadratic coefficient in the
Peano expansion at the origin.  We verify below that this coefficient is the
classical derivative of the gradient.  Summing the coefficient increments
from $k=0$ gives $|p|+|A|\le C$, uniformly over all normalized pairs.

It remains to pass from the pointwise expansion to a Hessian estimate. Let
$v\in\mathcal H_G(B_1)$ satisfy $\osc_{B_1}v\le1$, and subtract a constant
so that $\|v\|_{L^\infty(B_1)}\le1$.  For $x\in B_{1/2}$ set
\[
 w_x(y)\coloneqq v(x+\tfrac14y)-v(x),
 \qquad y\in B_1.
\]
Then $\|w_x\|_{L^\infty(B_1)}\le1$, and $w_x$ solves the equation for
$T_{1/16}G=\mathcal S_{0,1/16}G\in\mathfrak F$. Applying
\eqref{eq:c2-pointwise-expansion} to $w_x$ and returning to the original
variables yields coefficients $p_x\in\Rn$ and $A_x\in\Sn$, and constants
$r_0,C>0$, independent of $G,v$, and $x$, such that
\begin{equation}
 \left|v(z)-v(x)-p_x\cdot(z-x)
 -\frac12(z-x)\cdot A_x(z-x)\right|
 \le C|z-x|^{2+\alpha_*}
 \label{eq:c2-uniform-local-expansion}
\end{equation}
whenever $x\in B_{1/2}$ and $|z-x|\le r_0$.  The coefficient bounds in the
origin construction also give
\begin{equation}
 |p_x|+|A_x|\le C
 \qquad(x\in B_{1/2}).
 \label{eq:c2-uniform-jet-bounds}
\end{equation}
The first-order part of \eqref{eq:c2-uniform-local-expansion} shows already
that $p_x=Dv(x)$.

Let $x,y\in B_{1/2}$ and set $d\coloneqq|x-y|$. The case $d=0$ is
immediate. First assume $0<d\le r_0$, and
let $Q_x,Q_y$ be the quadratic polynomials in
\eqref{eq:c2-uniform-local-expansion}. Both
expansions hold on $B_{d/2}((x+y)/2)$, hence
\[
 \|Q_x-Q_y\|_{L^\infty(B_{d/2}((x+y)/2))}
 \le Cd^{2+\alpha_*}.
\]
By \Cref{lem:quadratic-coefficients},
\[
 |A_x-A_y|\le Cd^{\alpha_*}
 \quad\text{and}\quad
 |D(Q_x-Q_y)((x+y)/2)|\le Cd^{1+\alpha_*}.
\]
Writing $h=y-x$ and using
\[
 D(Q_x-Q_y)((x+y)/2)
 =p_x-p_y+\tfrac12(A_x+A_y)h,
\]
we obtain
\[
 |p_y-p_x-A_xh|
 \le C d^{1+\alpha_*}.
\]
Thus $Dv=p$ is differentiable at every point, with derivative $A_x$.
Consequently $A_x=D^2v(x)$, and the first estimate above proves that the
Hessian is locally $\alpha_*$-H\"older continuous.
For $d>r_0$, the same Hessian estimate follows from
\eqref{eq:c2-uniform-jet-bounds}, after increasing $C$.  Thus the normalized
interior $C^{2,\alpha_*}$ estimate holds.  Translation, spatial rescaling,
and normalization by the oscillation prove
\eqref{eq:qualitative-c2-estimate}.
\end{proof}

\begin{corollary}
\label{cor:full-class-c2-equivalence}
Fix $n\ge2$ and $0<\lambda\le\Lambda$.  The following assertions are
equivalent.
\begin{enumerate}
\item Every local viscosity solution of every
  \(G\in\mathcal E_{n,\lambda,\Lambda}\) admits a quadratic Taylor
  expansion at every point.
\item Every such solution belongs to $C^2_{\mathrm{loc}}$.
\item There are $\alpha=\alpha(n,\lambda,\Lambda)>0$ and
  $C=C(n,\lambda,\Lambda)<\infty$ such that, for every ball
  $B_R(x_0)$, every $G\in\mathcal E_{n,\lambda,\Lambda}$, and every
  $u\in\mathcal H_G(B_R(x_0))$,
  \begin{equation}
   R\|Du\|_{L^\infty(B_{R/2}(x_0))}
   +R^2\|D^2u\|_{L^\infty(B_{R/2}(x_0))}
   +R^{2+\alpha}[D^2u]_{C^{0,\alpha}(B_{R/2}(x_0))}
   \le C\osc_{B_R(x_0)}u.
   \label{eq:full-class-c2-estimate}
  \end{equation}
\end{enumerate}
\end{corollary}

\begin{proof}
The class $\mathcal E_{n,\lambda,\Lambda}$ is compact in the locally
uniform topology and is closed under every centered renormalization
$\mathcal S_{A,t}$ with $G(A)=0$.  Thus (1) implies (3) by
\Cref{thm:qualitative-c2-self-improvement}.  The implications
$(3)\Rightarrow(2)$ and $(2)\Rightarrow(1)$ are immediate.
\end{proof}

\subsection{Preservation of a prescribed exponent}
\label{sec:exponent-preservation}

The same compactness argument retains any prescribed qualitative H\"older
exponent.

\begin{theorem}
\label{thm:second-order-banach-steinhaus}
Let \(0<\gamma\leq1\), and let
\(\mathfrak F\subset\mathcal E_{n,\lambda,\Lambda}\) be nonempty, compact,
and satisfy \eqref{eq:centered-forward-invariance}.  Assume every local
solution of every equation in \(\mathfrak F\) belongs to
\(C^{2,\gamma}_{\mathrm{loc}}\).  Then, for every \(0<\beta<\gamma\), there is
\(C_{\beta,\mathfrak F}<\infty\) such that
\begin{equation}
 R\|Du\|_{L^\infty(B_{R/2}(x_0))}
 +R^2\|D^2u\|_{L^\infty(B_{R/2}(x_0))}
 +R^{2+\beta}[D^2u]_{C^{0,\beta}(B_{R/2}(x_0))}
 \le C_{\beta,\mathfrak F}\osc_{B_R(x_0)}u
 \label{eq:second-order-family-estimate}
\end{equation}
for every \(G\in\mathfrak F\) and
\(u\in\mathcal H_G(B_R(x_0))\).
\end{theorem}

\begin{proof}
The proof follows the preceding argument, with the decay now calibrated to
$\beta$. For a normalized pair $(G,w)$, let $\pi_w$ be the quadratic Taylor
polynomial of $w$ at the origin. Then $G(D^2\pi_w)=0$, and, since
$\gamma>\beta$, there is $\rho_w\in(0,1/4)$ such that
\[
 \|w-\pi_w\|_{L^\infty(B_{\rho_w})}
 <\frac18\rho_w^{2+\beta}.
\]
If \(H\) is close to \(G\), the scalar correction from
\Cref{lem:zero-level-correction} changes \(D^2\pi_w\) by \(o(1)I\) and
places it on \(H^{-1}(0)\). Strictness of the last inequality and compactness
of $\mathcal P(\mathfrak F)$ therefore give numbers
\(\theta\in(0,1/4)\) and \(K<\infty\) with the following property: for every
normalized pair \((H,z)\), there are \(r\in[\theta,1/4]\) and a quadratic
polynomial
\[
 q(x)=a+b\cdot x+\frac12x\cdot Bx
\]
such that
\begin{equation}
 H(B)=0,\qquad
 \|z-q\|_{L^\infty(B_r)}\le\frac12r^{2+\beta},
 \qquad |a|+|b|+|B|\le K.
 \label{eq:quadratic-one-step}
\end{equation}

Fix $G\in\mathfrak F$ and $v\in\mathcal H_G(B_1)$ with
$\|v\|_{L^\infty(B_1)}\le1$. Put $R_0=1$, $P_0=0$, and
$E_k\coloneqq R_k^\beta$. If
\(P_k(x)=c_k+p_k\cdot x+\frac12x\cdot A_kx\), set
\[
 v_k(x)\coloneqq\frac{v(R_kx)-P_k(R_kx)}{E_kR_k^2},
 \qquad
 G_k\coloneqq\mathcal S_{A_k,E_k^{-1}}G.
\]
Inductively, we require
\begin{equation}
 \|v-P_k\|_{L^\infty(B_{R_k})}\le E_kR_k^2,
 \qquad G(A_k)=0,
 \qquad \|v_k\|_{L^\infty(B_1)}\le1.
 \label{eq:prescribed-exponent-invariants}
\end{equation}
These conditions hold at $k=0$. Since $G_k\in\mathfrak F$, apply
\eqref{eq:quadratic-one-step} to \(v_k\), obtaining \(r_k\) and \(q_k\),
and define
\[
 R_{k+1}\coloneqq R_kr_k,\qquad E_{k+1}\coloneqq E_kr_k^\beta,
 \qquad
 P_{k+1}(y)\coloneqq P_k(y)+E_kR_k^2q_k(y/R_k).
\]
Writing $B_k=D^2q_k$, the next normalized function is
\[
 v_{k+1}(x)
 =\frac{v_k(r_kx)-q_k(r_kx)}{r_k^{2+\beta}},
 \qquad \|v_{k+1}\|_{L^\infty(B_1)}\le\frac12.
\]
Since $G_k(B_k)=0$, the composition rule
\eqref{eq:centered-composition} gives
\[
 G_{k+1}
 =\mathcal S_{B_k,r_k^{-\beta}}G_k
 =\mathcal S_{D^2P_{k+1},E_{k+1}^{-1}}G\in\mathfrak F,
\]
and $v_{k+1}$ solves the equation for $G_{k+1}$. We also have
$G(D^2P_{k+1})=0$. Rescaling the error gives the first condition in
\eqref{eq:prescribed-exponent-invariants} at index $k+1$. Thus the
construction continues. Since $E_k=R_k^\beta$, the coefficient increments
satisfy
\begin{align*}
 |D^2P_{k+1}-D^2P_k|&\le KR_k^\beta,\\
 |DP_{k+1}(0)-DP_k(0)|&\le KR_k^{1+\beta},\\
 |P_{k+1}(0)-P_k(0)|&\le KR_k^{2+\beta}.
\end{align*}
They are summable. Hence $P_k$ converges to a quadratic polynomial
$P_\infty$, with
\begin{align*}
 |D^2(P_\infty-P_k)|&\le CR_k^\beta,\\
 |D(P_\infty-P_k)(0)|&\le CR_k^{1+\beta},\\
 |(P_\infty-P_k)(0)|&\le CR_k^{2+\beta}.
\end{align*}
The error at the origin shows that $P_\infty(0)=v(0)$. Given
$0<r<\theta$, choose $k$ so that $R_{k+1}\le r<R_k$. The induction
estimate, these tail bounds, and $R_{k+1}\ge\theta R_k$ yield
\[
 \|v-P_\infty\|_{L^\infty(B_r)}\le Cr^{2+\beta}
 \qquad(0<r<\theta),
\]
with uniform bounds for its linear and quadratic coefficients. Thus
$P_\infty$ is the quadratic Taylor polynomial of $v$ at the origin.

Apply this estimate after recentering at every \(x\in B_{1/2}\).  Comparing
the two quadratic polynomials on
\(B_{|x-y|/2}((x+y)/2)\) and using
\Cref{lem:quadratic-coefficients} gives
\[
 |D^2v(x)-D^2v(y)|\le C|x-y|^\beta.
\]
The gradient part of the same coefficient estimate gives the jet-coherence
relation proved in \Cref{thm:qualitative-c2-self-improvement}, and hence the
quadratic coefficients are the classical Hessian coefficients.
The coefficient bounds give the corresponding \(L^\infty\) estimates for
\(Dv\) and \(D^2v\).  Translation, spatial rescaling, and normalization by
the oscillation yield \eqref{eq:second-order-family-estimate}.
\end{proof}

\begin{corollary}
\label{cor:centered-hull-criterion}
Let \(F\in\mathcal E_{n,\lambda,\Lambda}\).  If every local solution of
every equation in \(\mathcal H^{(2)}(F)\) is twice differentiable
everywhere, then solutions of \(F(D^2u)=0\) satisfy
\eqref{eq:qualitative-c2-estimate} for some \(\alpha_*>0\).  If all these
solutions belong to \(C^{2,\gamma}_{\mathrm{loc}}\), then
\eqref{eq:second-order-family-estimate} holds for every \(\beta<\gamma\).
\end{corollary}

\begin{proof}
Apply
\Cref{thm:qualitative-c2-self-improvement,thm:second-order-banach-steinhaus}
to the compact centered-invariant family given by \Cref{prop:centered-hull}.
\end{proof}

\begin{corollary}[Equality of Hessian thresholds]
\label{cor:second-order-thresholds}
Let \(\mathfrak F\) be a nonempty compact family satisfying
\eqref{eq:centered-forward-invariance}, and suppose its qualitative
\(C^{2,\alpha}\) range is nonempty.  Then
\[
\begin{aligned}
 &\sup\{\alpha\in(0,1]:\mathcal H_G(B)\subset C^{2,\alpha}_{\mathrm{loc}}(B)
       \text{ for every }G\in\mathfrak F\}\\
 &\qquad=
 \sup\{\alpha\in(0,1]:\eqref{eq:second-order-family-estimate}
       \text{ holds uniformly over }\mathfrak F\}.
\end{aligned}
\]
\end{corollary}

\begin{proof}
Every estimate gives the corresponding qualitative regularity.  Conversely,
if \(\beta\) lies below the qualitative supremum, choose
\(\gamma>\beta\) in the qualitative range and apply
\Cref{thm:second-order-banach-steinhaus}.
\end{proof}

\section{Tangent equations and partial regularity}
\label{sec:partial-regularity}

The full centered hull is the correct object for a global $C^{2,\alpha}$
theory.  Partial regularity requires less.  Near a point where a solution has
a quadratic expansion, the error is $o(r^2)$; after division by its amplitude,
the centered parameter tends only in the direction $t\to\infty$.  We show
that qualitative regularity on this one-sided tail is enough to recover the
flat-solution input in the Armstrong--Silvestre--Smart argument
\cite{ArmstrongSilvestreSmart2012}.  No differentiability of the original
operator is required.

\subsection{From tangent equations to flatness}

For $F\in\mathcal E_{n,\lambda,\Lambda}$ and $R>0$, define
\begin{equation}
 \mathcal K_R^{(2)}(F)
 \coloneqq
 \overline{\left\{
   \mathcal S_{A,t}F:F(A)=0,\ t\ge R
 \right\}}^{\,C_{\mathrm{loc}}(\Sn)}
 \label{eq:centered-tail}
\end{equation}
and the \emph{one-sided centered tangent hull}
\begin{equation}
 \mathfrak T_\infty^{(2)}(F)
 \coloneqq\bigcap_{R>0}\mathcal K_R^{(2)}(F).
 \label{eq:centered-tangent-hull}
\end{equation}
Equivalently, $H\in\mathfrak T_\infty^{(2)}(F)$ if and only if
\[
 \mathcal S_{A_j,t_j}F\longrightarrow H,
 \qquad F(A_j)=0,\qquad t_j\longrightarrow\infty,
\]
after passage to a suitable sequence.  The arbitrary zero-level centers are
essential: they are generated when the quadratic part of a solution is
subtracted.

\begin{proposition}
\label{prop:centered-tangent-hull}
The family $\mathfrak T_\infty^{(2)}(F)$ is nonempty and compact.  Moreover:
\begin{enumerate}
\item if $H\in\mathcal K_R^{(2)}(F)$, $H(B)=0$, and $s>0$, then
\begin{equation}
 \mathcal S_{B,s}H\in\mathcal K_{sR}^{(2)}(F);
 \label{eq:centered-tail-covariance}
\end{equation}
in particular, $\mathfrak T_\infty^{(2)}(F)$ satisfies the centered
invariance \eqref{eq:centered-forward-invariance};
\item for every \(s>0\),
\begin{equation}
 \mathfrak T_\infty^{(2)}(T_sF)
 =\mathfrak T_\infty^{(2)}(F).
 \label{eq:tangent-hull-amplitude-invariance}
\end{equation}
\end{enumerate}
\end{proposition}

\begin{proof}
The sets $\mathcal K_R^{(2)}(F)$ are nonempty compact subsets of the compact
operator space $\mathcal E_{n,\lambda,\Lambda}$ and decrease as $R$
increases.  Their intersection is therefore nonempty and compact.

To prove \eqref{eq:centered-tail-covariance}, first let
\[
 H_j=\mathcal S_{A_j,t_j}F\longrightarrow H,
 \qquad F(A_j)=0,\qquad t_j\ge R.
\]
If $H(B)=0$, set $B_j\coloneqq B+\tau_B(H_j)I$.  By
\Cref{lem:zero-level-correction},
\[
 H_j(B_j)=0,
 \qquad B_j\longrightarrow B.
\]
The composition law \eqref{eq:centered-composition} gives
\begin{equation}
 \mathcal S_{B_j,s}H_j
 =\mathcal S_{A_j+t_j^{-1}B_j,st_j}F.
 \label{eq:centered-tail-composition}
\end{equation}
Since $H_j(B_j)=0$, the new center satisfies
$F(A_j+t_j^{-1}B_j)=F(A_j)=0$.  The right-hand side of
\eqref{eq:centered-tail-composition} thus belongs to the orbit defining
$\mathcal K_{sR}^{(2)}(F)$.  Joint continuity of the centered action and
$B_j\to B$ prove \eqref{eq:centered-tail-covariance}.  Applying this with
$R$ replaced by $R/s$ yields centered invariance of the intersection.

Finally, $T_sF=\mathcal S_{0,s}F$, and
\[
 \mathcal S_{A,t}(T_sF)
 =\mathcal S_{s^{-1}A,ts}F,
 \qquad
 (T_sF)(A)=0\Longleftrightarrow F(s^{-1}A)=0.
\]
Letting $t\to\infty$ proves
\eqref{eq:tangent-hull-amplitude-invariance}.
\end{proof}

The direction $t\to\infty$ is forced by the normalization. If $P$ is the
quadratic Taylor polynomial at $x_0$, then
$a(r)\coloneqq\|u-P\|_{L^\infty(B_r(x_0))}=o(r^2)$ and the rescaled error is
governed by $\mathcal S_{D^2P,r^2/a(r)}F$. Thus $r^2/a(r)\to\infty$
whenever $a(r)>0$; if $a(r)=0$ for some $r$, the solution already agrees
with $P$ in a neighborhood. The tangent hull must also be formed directly
from $F$. Taking a one-sided tail only
after forming the full centered hull would recover the full hull, since its
backward renormalizations can cancel any prescribed lower bound on $t$.

\begin{proposition}
\label{prop:tangent-hull-flat-solutions}
Suppose that, for some $0<\alpha\le1$ and $C_0<\infty$, every
$H\in\mathfrak T_\infty^{(2)}(F)$ and every
$w\in\mathcal H_H(B_1)$ satisfy
\begin{equation}
 \|Dw\|_{L^\infty(B_{1/2})}
 +\|D^2w\|_{L^\infty(B_{1/2})}
 +[D^2w]_{C^{0,\alpha}(B_{1/2})}
 \le C_0\osc_{B_1}w.
 \label{eq:tangent-hull-uniform-estimate}
\end{equation}
Then, for every $0<\beta<\alpha$, there are
$\delta_\beta>0$ and $C_\beta<\infty$ with the following property.  If
$u\in\mathcal H_F(B_R(x_0))$ and $P$ is a quadratic polynomial satisfying
$F(D^2P)=0$ and
\begin{equation}
 \|u-P\|_{L^\infty(B_R(x_0))}\le\delta R^2,
 \qquad 0<\delta\le\delta_\beta,
 \label{eq:flat-solution-hypothesis}
\end{equation}
then $u\in C^{2,\beta}(B_{R/2}(x_0))$ and
\begin{equation}
 \|D^2u-D^2P\|_{L^\infty(B_{R/2}(x_0))}
 +R^\beta[D^2u]_{C^{0,\beta}(B_{R/2}(x_0))}
 \le C_\beta\delta.
 \label{eq:flat-solution-estimate}
\end{equation}
\end{proposition}

\begin{proof}
We first obtain a quadratic improvement uniformly in the deep centered tail.
Choose $\rho\in(0,1/4)$ so small that the Taylor remainder supplied by
\eqref{eq:tangent-hull-uniform-estimate} is at most
$\frac14\rho^{2+\beta}$.  We claim that there exists $R_0<\infty$ such that,
whenever
\[
 G=\mathcal S_{A,t}F,\qquad F(A)=0,\qquad t\ge R_0,
\]
and $v\in\mathcal H_G(B_1)$ satisfies $\|v\|_{L^\infty(B_1)}\le1$, there is
a quadratic polynomial $q$ such that
\begin{equation}
 G(D^2q)=0,\qquad
 \|v-q\|_{L^\infty(B_\rho)}\le\rho^{2+\beta},
 \qquad \|q\|_{L^\infty(B_1)}\le C.
 \label{eq:deep-tail-one-step}
\end{equation}

Otherwise there are bad pairs $(G_j,v_j)$ with tail parameters
$t_j\to\infty$.  Operator compactness, interior H\"older estimates, and
viscosity stability give, after passage to a subsequence,
\[
 G_j\longrightarrow H\in\mathfrak T_\infty^{(2)}(F),
 \qquad v_j\longrightarrow v,
 \qquad H(D^2v)=0.
\]
Let $q$ be the quadratic Taylor polynomial of $v$ at the origin.  By
\eqref{eq:tangent-hull-uniform-estimate}, the choice of $\rho$, and
\Cref{lem:quadratic-compatibility},
\[
 H(D^2q)=0,\qquad
 \|v-q\|_{L^\infty(B_\rho)}
 \le\tfrac14\rho^{2+\beta}.
\]
Correct $D^2q$ by the scalar
$\tau_{D^2q}(G_j)I$ and retain the constant and linear coefficients of $q$.
The corrected polynomials $q_j$ converge to $q$ and satisfy
$G_j(D^2q_j)=0$.  Uniform convergence then gives
\eqref{eq:deep-tail-one-step} for large $j$, the desired contradiction.

The estimate for normalized solutions of every sufficiently deep tail
operator now follows by the fixed-radius quadratic iteration used in the
proof of \Cref{thm:second-order-banach-steinhaus}.  Indeed, after one step the
normalized residual is governed by
\[
 \mathcal S_{D^2q,\rho^{-\beta}}G
 =\mathcal S_{A+t^{-1}D^2q,t\rho^{-\beta}}F.
\]
The new center remains on $F^{-1}(0)$ because $G(D^2q)=0$, while the tail
parameter increases.  Thus every iterate remains in the range where
\eqref{eq:deep-tail-one-step} is valid.  Set $R_1\coloneqq32R_0$.  For an initial
parameter $t\ge R_1$, the interior estimate follows as well.  Indeed, for
$x\in B_{1/2}$ apply the
pointwise iteration to
\[
 w_x(y)\coloneqq\frac{v(x+\tfrac14y)-v(x)}2,
 \qquad y\in B_1.
\]
This is normalized and is governed by
$\mathcal S_{0,1/32}G=\mathcal S_{A,t/32}F$.  Thus $t\ge R_1$ keeps the
localized equation in the controlled tail.  Returning to the original
variables and comparing the quadratic jets at two nearby centers exactly as
in the proof of \Cref{thm:qualitative-c2-self-improvement} gives both the
Hölder estimate for the quadratic coefficients and their compatibility with
the linear jets.  Consequently,
\begin{equation}
 \|Dv\|_{L^\infty(B_{1/2})}
 +\|D^2v\|_{L^\infty(B_{1/2})}
 +[D^2v]_{C^{0,\beta}(B_{1/2})}
 \le C_\beta
 \label{eq:deep-tail-c2beta}
\end{equation}
for every normalized solution of $\mathcal S_{A,t}F$ with $t\ge R_1$.

Return to \eqref{eq:flat-solution-hypothesis}, first with $R=1$ and $x_0=0$, and put $A=D^2P$. If $a\coloneqq\|u-P\|_{L^\infty(B_1)}>0$, then
\[
    v\coloneqq\frac{u-P}{a}
\]
is normalized and solves the equation for $\mathcal S_{A,a^{-1}}F$. Taking $\delta_\beta\le R_1^{-1}$ and applying
\eqref{eq:deep-tail-c2beta} gives \eqref{eq:flat-solution-estimate}, with $a$ in place of $\delta$ and hence also with $\delta$. The case $a=0$ is immediate. Translation and scaling prove the general statement.
\end{proof}

\subsection{From flatness to partial regularity}

We recall the universal input from partial regularity.  For
$u\in C(B_1)$ and $x\in B_1$, let $\Psi(u,B_1)(x)$ be the infimum of the
numbers $K\ge0$ for which there are $p\in\Rn$ and $A\in\Sn$ satisfying
\begin{equation}
 \left|u(y)-u(x)-p\cdot(y-x)
 -\frac12(y-x)\cdot A(y-x)\right|
 \le\frac K6|y-x|^3
 \quad(y\in B_1).
 \label{eq:cubic-contact-modulus}
\end{equation}
Armstrong--Silvestre--Smart prove that there are
$C<\infty$ and $\varepsilon=\varepsilon(n,\lambda,\Lambda)>0$ such that,
for every $G\in\mathcal E_{n,\lambda,\Lambda}$ and every
$u\in\mathcal H_G(B_1)$ with $\|u\|_{L^\infty(B_1)}\le1$,
\begin{equation}
 \left|\left\{x\in B_{1/2}:
 \Psi(u,B_1)(x)>t\right\}\right|
 \le Ct^{-\varepsilon}
 \qquad(t>1).
 \label{eq:universal-w3epsilon}
\end{equation}
This is the viscosity $W^{3,\varepsilon}$ estimate. Although
\cite[Lemma~5.2]{ArmstrongSilvestreSmart2012} is stated under that paper's
standing assumptions, its proof does not use differentiability of the
operator.  Indeed, the universal interior $C^{1,\alpha}$ estimate first gives
$u\in C^1$.  Difference quotients of $u$ in any fixed direction satisfy the
two Pucci inequalities by
\cite[Proposition~5.5]{CaffarelliCabre1995}; passing to the directional
derivative by viscosity stability preserves those inequalities.  Applying
the universal $W^{2,\varepsilon}$ estimate to the coordinate derivatives and
then \cite[Lemma~5.1]{ArmstrongSilvestreSmart2012} gives
\eqref{eq:universal-w3epsilon}.  Thus this input requires only uniform
ellipticity. Differentiability of the operator enters the original partial
regularity argument through the flat-solution theorem, which
\Cref{prop:tangent-hull-flat-solutions} replaces.

\begin{theorem}[Twice differentiability points are regular]
\label{thm:tangent-hull-partial-regularity}
Let \(F\in\mathcal E_{n,\lambda,\Lambda}\).  Assume every local solution of
every equation in \(\mathfrak T_\infty^{(2)}(F)\) is twice differentiable at
every point.  Then there is
\(\alpha_{\mathrm{tan}}=\alpha_{\mathrm{tan}}(\mathfrak T_\infty^{(2)}(F))>0\) with
the following property.  If \(u\in\mathcal H_F(\Omega)\), then
\begin{equation}
 \Reg(u)\coloneqq\{x\in\Omega:u\text{ is twice differentiable at }x\}
 \label{eq:regular-set-definition}
\end{equation}
is open and
\begin{equation}
 u\in C^{2,\beta}_{\mathrm{loc}}(\Reg(u))
 \qquad(0<\beta<\alpha_{\mathrm{tan}}).
 \label{eq:tangent-hull-regular-set}
\end{equation}
Consequently,
\[
 \Sing(u)\coloneqq\Omega\setminus\Reg(u)
\]
is exactly the set of points having no neighborhood on which \(u\) is
\(C^2\).  It is relatively closed, and
\begin{equation}
 \mathcal H^{n-\varepsilon}(\Sing(u)\cap K)<\infty
 \qquad(K\Subset\Omega),
 \label{eq:tangent-hull-partial-regularity}
\end{equation}
where \(\varepsilon=\varepsilon(n,\lambda,\Lambda)>0\) is universal.

If, more strongly, all tangent-equation solutions belong to
\(C^{2,\gamma}_{\mathrm{loc}}\), then \eqref{eq:tangent-hull-regular-set} holds
for every \(\beta<\gamma\).
\end{theorem}

\begin{proof}
\Cref{prop:centered-tangent-hull} shows that the tangent family is
compact and centered-invariant. \Cref{thm:qualitative-c2-self-improvement}
therefore gives a uniform \(C^{2,\alpha_{\mathrm{tan}}}\) estimate on that family.
\Cref{prop:tangent-hull-flat-solutions} then supplies, for each
\(\beta<\alpha_{\mathrm{tan}}\), numbers \(\delta_\beta>0\) and \(C_\beta\)
such that
\begin{equation}
 \|u-P\|_{L^\infty(B_r(x))}\le\delta_\beta r^2,\qquad
 F(D^2P)=0,
 \label{eq:epsilon-regularity-trigger}
\end{equation}
implies \(u\in C^{2,\beta}(B_{r/2}(x))\).

We first identify the regular set.  Suppose \(u\) is twice differentiable at
\(x\), and let \(P_x\) be its quadratic Taylor polynomial. By
\Cref{lem:quadratic-compatibility}, \(F(D^2P_x)=0\), while
\[
 a(r)\coloneqq\|u-P_x\|_{L^\infty(B_r(x))}=o(r^2).
\]
For sufficiently small \(r\), condition \eqref{eq:epsilon-regularity-trigger}
holds.  Hence \(u\) is \(C^{2,\beta}\) in a neighborhood of \(x\).  The
converse is immediate.  Taking one fixed
\(\beta_0\in(0,\alpha_{\mathrm{tan}})\) proves that \(\Reg(u)\) is open and that
its complement agrees with the classical singular set.  Repeating the same
argument at each point, with any other \(\beta<\alpha_{\mathrm{tan}}\), proves
\eqref{eq:tangent-hull-regular-set} on $\Reg(u)$.

It remains to estimate its complement.  Normalize first in \(B_1\), with
\(\|u\|_{L^\infty(B_1)}\le1\).  The flatness theorem gives
\(\delta_0>0\) with the following consequence.  If
\(y\in\overline B_{1/4}\), \(0<r<1/16\), and some \(z\in B_r(y)\) satisfies
\begin{equation}
 \Psi(u,B_1)(z)<\delta_0r^{-1},
 \label{eq:good-cubic-contact}
\end{equation}
then \(y\in\Reg(u)\).  Indeed, an admissible quadratic polynomial \(P_z\)
in \eqref{eq:cubic-contact-modulus}, with constant smaller than
\(\delta_0r^{-1}\), satisfies
\(F(D^2P_z)=0\) and
\[
 \|u-P_z\|_{L^\infty(B_{4r}(z))}
 \le C\delta_0r^2.
\]
After decreasing \(\delta_0\) by a universal factor, the flatness theorem
applies in \(B_{4r}(z)\) and gives regularity on \(B_{2r}(z)\), which
contains \(y\).

Choose a maximal disjoint family
\(\{B_r(y_i)\}\), with
\(y_i\in\Sing(u)\cap\overline B_{1/4}\), where
$0<r<\min\{1/16,\delta_0/2\}$. The balls \(B_{3r}(y_i)\) cover
this set.  By the contrapositive of \eqref{eq:good-cubic-contact},
\[
 \bigcup_iB_r(y_i)
 \subset\{\Psi(u,B_1)\ge\delta_0r^{-1}\}
 \subset\{\Psi(u,B_1)>\tfrac12\delta_0r^{-1}\}.
\]
The universal estimate \eqref{eq:universal-w3epsilon} gives
\[
 \#\{i\}\,r^n\le C_{\delta_0}r^\varepsilon.
\]
Thus \(\sum_i(3r)^{n-\varepsilon}\le C\), uniformly as \(r\downarrow0\),
so the \((n-\varepsilon)\)-dimensional Hausdorff contents at scale \(6r\)
are uniformly bounded.  Letting \(r\downarrow0\) proves the local Hausdorff
estimate in the normalized ball.  Covering,
spatial rescaling, and amplitude normalization give
\eqref{eq:tangent-hull-partial-regularity} in a general domain.
The tangent family is unchanged by amplitude normalization, by
\eqref{eq:tangent-hull-amplitude-invariance}.

Under the stronger \(C^{2,\gamma}\) hypothesis, apply
\Cref{thm:second-order-banach-steinhaus} on the tangent family with an
intermediate exponent between \(\beta\) and \(\gamma\).  The same proof
then gives every \(\beta<\gamma\).
\end{proof}

\begin{remark}
The same covering argument gives, for \(K\Subset\Omega\) and all sufficiently
small \(r\),
\[
 \bigl|\{x:\dist(x,\Sing(u)\cap K)<r\}\bigr|
 \le C_{K,u,F}r^\varepsilon.
\]
Thus every compact part of the singular set has upper Minkowski dimension at
most \(n-\varepsilon\).
\end{remark}

\begin{corollary}
\label{cor:linear-tangent-equations}
Let $F\in\mathcal E_{n,\lambda,\Lambda}$. Assume every equation represented by an operator in
\(\mathfrak T_\infty^{(2)}(F)\) is equivalent to a constant-coefficient
linear uniformly elliptic equation.  Then the conclusions of
\Cref{thm:tangent-hull-partial-regularity} hold for every
\(0<\beta<1\).
\end{corollary}

\begin{proof}
Solutions of a constant-coefficient linear uniformly elliptic equation
satisfy interior $C^{2,\gamma}$ estimates for every $\gamma<1$. Apply
\Cref{thm:tangent-hull-partial-regularity}.
\end{proof}

The corollary concerns the equation \(F=0\), rather than the particular
function used to represent it.  To compare it with the criterion of
Armstrong--Silvestre--Smart, suppose that \(F\in C^1(\Sn)\) and that \(DF\)
has a global modulus of continuity \(\omega\).  Writing \(DF(A)[M]\) for the
derivative of \(F\) at \(A\) in the direction \(M\), the fundamental theorem
of calculus gives
\[
 \mathcal S_{A,t}F(M)-DF(A)[M]
 =
 \int_0^1
 \bigl(DF(A+\theta t^{-1}M)-DF(A)\bigr)[M]\,d\theta.
\]
Consequently,
\[
 \left|\mathcal S_{A,t}F(M)-DF(A)[M]\right|
 \le |M|\,\omega(t^{-1}|M|).
\]
The convergence to the linearization is therefore uniform in the center
\(A\) on bounded sets of matrices.  Since uniform ellipticity keeps the
coefficients \(DF(A)\) in a compact set, every centered tangent equation is
a constant-coefficient linear equation.

\Cref{cor:linear-tangent-equations} starts directly from this last
conclusion.  It requires neither differentiability of \(F\) nor a prescribed
rate of convergence to the tangent equations.  In fact, a tangent operator
need not itself be linear: it is enough that its zero equation agree with a
linear one. Thus a tangent operator already known to lie in
$\mathcal E_{n,\lambda,\Lambda}$ may have the form $\eta\circ L$, where
$L$ is linear and $\eta$ is strictly increasing with $\eta(0)=0$; no
differentiability of $\eta$ is required.

\section{Scalar finite-switching laws}
\label{sec:switching-equations}

We conclude with a class of nonsmooth equations for which the tangent
criterion can be checked explicitly. For $P\in\Sn$, write
\[
 L_P(M)\coloneqq\operatorname{tr}(PM)
\]
for the corresponding linear operator. For a finite switching law, the
tangent criterion can be read off directly. The original operator may pass
through several elliptic regimes, while a centered blow-up
can retain at most one switching level.  Its tangent equations are therefore
linear or have a single hinge; in the latter case they are convex or concave.

\begin{corollary}[Finite switching nonlinearities]
\label{cor:finite-switching}
Let \(A,B\in\Sn\), and let \(\phi:\mathbb R\to\mathbb R\) be continuous and
piecewise affine.  Suppose that its finitely many affine pieces have slopes
\(m_0,\ldots,m_k\), and that
\begin{equation}
 \lambda I\le A+m_iB\le\Lambda I
 \qquad(0\le i\le k).
 \label{eq:switching-ellipticity}
\end{equation}
Define
\begin{equation}
 F(M)\coloneqq L_A(M)+\phi\bigl(L_B(M)\bigr)-\phi(0).
 \label{eq:switching-operator}
\end{equation}
Then \(F\in\mathcal E_{n,\lambda,\Lambda}\).  Moreover, there are universal
numbers
\[
 \alpha_{\mathrm{EK}}=\alpha_{\mathrm{EK}}(n,\lambda,\Lambda)>0,
 \qquad
 \varepsilon=\varepsilon(n,\lambda,\Lambda)>0,
\]
with the following property.  If \(u\in\mathcal H_F(\Omega)\), then
\(\Reg(u)\) is open and
\[
 u\in C^{2,\beta}_{\mathrm{loc}}(\Reg(u))
 \qquad(0<\beta<\alpha_{\mathrm{EK}}).
\]
The complementary set \(\Sing(u)\) is relatively closed and
\[
 \mathcal H^{n-\varepsilon}(\Sing(u)\cap K)<\infty
 \qquad(K\Subset\Omega).
\]
\end{corollary}

\begin{proof}
We first verify ellipticity.  Let \(N\ge0\).  If \(L_B(N)\ne0\), the secant
slope
\[
 \sigma\coloneqq
 \frac{\phi(L_B(M+N))-\phi(L_B(M))}{L_B(N)}
\]
belongs to the convex hull of \(\{m_0,\ldots,m_k\}\).  If \(L_B(N)=0\),
the same calculation holds with any \(\sigma\) in this convex hull.  In
either case,
\[
 F(M+N)-F(M)=L_{A+\sigma B}(N).
\]
Condition~\eqref{eq:switching-ellipticity} is preserved under convex
combinations, and hence
\[
 \lambda\operatorname{tr}N
 \le F(M+N)-F(M)
 \le\Lambda\operatorname{tr}N.
\]
Since \(F(0)=0\), this proves that
\(F\in\mathcal E_{n,\lambda,\Lambda}\).

It remains to identify the centered tangents.  Suppose
\[
 \mathcal S_{Q_j,t_j}F\longrightarrow H,
 \qquad F(Q_j)=0,
 \qquad t_j\longrightarrow\infty,
\]
and set \(s_j\coloneqq L_B(Q_j)\).  Directly from the definition,
\begin{equation}
 \mathcal S_{Q_j,t_j}F(M)
 =L_A(M)+t_j\!\left[
   \phi\bigl(s_j+t_j^{-1}L_B(M)\bigr)-\phi(s_j)
 \right].
 \label{eq:switching-centered-profile}
\end{equation}
Let \(\mathcal C\) be the finite set of breakpoints of \(\phi\).  On a
bounded set of matrices, the last term samples \(\phi\) only within an
interval of length \(O(t_j^{-1})\) around \(s_j\).

Assume first that
\[
 t_j\dist(s_j,\mathcal C)\longrightarrow\infty.
\]
After passage to a subsequence, this interval lies in one affine piece of
\(\phi\), with fixed slope \(m_i\).  Formula
\eqref{eq:switching-centered-profile} then gives
\[
 H(M)=L_{A+m_iB}(M),
\]
so the tangent equation is linear.

In the remaining case, a subsequence approaches one breakpoint
\(c\in\mathcal C\) at the scale \(t_j^{-1}\).  More precisely, for some
\(q\in\mathbb R\),
\[
 t_j(s_j-c)\longrightarrow q.
\]
Let \(m_-\) and \(m_+\) be the slopes immediately to the left and right of
\(c\).  Near this point,
\[
 \phi(c+h)=\phi(c)+m_-h+(m_+-m_-)h_+.
\]
Substitution into \eqref{eq:switching-centered-profile} yields
\begin{equation}
 H(M)=L_{A+m_-B}(M)
 +(m_+-m_-)
 \left[\bigl(q+L_B(M)\bigr)_+-q_+\right].
 \label{eq:switching-tangent}
\end{equation}
Thus \(H\) is convex when \(m_+\ge m_-\), concave when
\(m_+\le m_-\), and linear when the two slopes agree.

Every tangent operator retains the ellipticity constants
\(\lambda,\Lambda\). The Evans--Krylov theorem
\cite{Evans1982,Krylov1982}, in the viscosity form of
\cite[Theorem~6.6]{CaffarelliCabre1995}, supplies the same exponent
\(\alpha_{\mathrm{EK}}(n,\lambda,\Lambda)\) for all of them. The conclusion now
follows from \Cref{thm:tangent-hull-partial-regularity}.
\end{proof}

\subsection{A four-state Isaacs equation}

The corollary applies to a simple model in every dimension \(n\ge2\).  Set
\[
 B\coloneqq\operatorname{diag}(1,-1,0,\ldots,0),
 \qquad
 \ell(M)\coloneqq L_B(M)=M_{11}-M_{22},
\]
and define
\[
 \psi(s)\coloneqq2s_+-4(s-1)_++(s-2)_+.
\]
Equivalently,
\begin{equation}
 \psi(s)=
 \begin{cases}
  0,     &s\le0,\\
  2s,    &0\le s\le1,\\
  4-2s,  &1\le s\le2,\\
  2-s,   &s\ge2.
 \end{cases}
 \label{eq:four-state-profile}
\end{equation}
For \(0<\mu<1/2\), let
\begin{equation}
 F_\mu(M)\coloneqq\operatorname{tr}M+\mu\psi\bigl(\ell(M)\bigr).
 \label{eq:four-operator-isaacs}
\end{equation}
The four successive slopes of \(\psi\) are
\[
 0,\qquad 2,\qquad -2,\qquad -1.
\]
Because this sequence is neither nondecreasing nor nonincreasing,
\(F_\mu\) is neither convex nor concave.
Accordingly, the coefficient matrices of the affine pieces of \(F_\mu\) are
\[
 I,\qquad I+2\mu B,\qquad I-2\mu B,\qquad I-\mu B.
\]
Their eigenvalues lie between \(1-2\mu\) and \(1+2\mu\).  These are the
ellipticity constants of \(F_\mu\). The elementary identity
\[
 \psi(s)=
 \min\bigl\{\max\{0,2s\},\max\{4-2s,2-s\}\bigr\}
\]
exhibits \(F_\mu\) as an Isaacs operator.  Namely, if
\[
\begin{array}{c|cc}
 &j=1&j=2\\ \hline
 i=1
 &\mathcal L_{11}(M)=\operatorname{tr}M
 &\mathcal L_{12}(M)=\operatorname{tr}M+2\mu\ell(M)\\[1mm]
 i=2
 &\mathcal L_{21}(M)=\operatorname{tr}M+4\mu-2\mu\ell(M)
 &\mathcal L_{22}(M)=\operatorname{tr}M+2\mu-\mu\ell(M),
\end{array}
\]
then
\[
 F_\mu(M)=\min_{i=1,2}\max_{j=1,2}\mathcal L_{ij}(M).
\]
As \(\ell(M)\) crosses the levels \(0,1,2\), the four affine states become
active successively, and none is redundant.

Consider the two-dimensional affine plane
\[
 M(x,y)\coloneqq\frac{x}{n}I+\frac{y}{2}B.
\]
Since \(\operatorname{tr}B=0\) and \(\operatorname{tr}(B^2)=2\),
\[
 \operatorname{tr}M(x,y)=x,
 \qquad
 \ell(M(x,y))=y.
\]
The zero set of \(F_\mu\) therefore cuts this plane along the polygonal
graph
\begin{equation}
 x=-\mu\psi(y).
 \label{eq:four-state-zero-graph}
\end{equation}

This graph rules out the Cabr\'e--Caffarelli structure even at the level of
the equation \(F=0\).  Indeed, suppose that an operator
\(G=\min\{C,V\}\), with \(C\) concave and \(V\) convex and both uniformly
elliptic, had the same zero set.  Uniform ellipticity in the \(x\)-direction
gives unique functions \(c,v:\mathbb R\to\mathbb R\) such that
\[
 C(M(c(y),y))=0,
 \qquad
 V(M(v(y),y))=0.
\]
The superlevel sets of \(C\) show that \(c\) is convex, while the sublevel
sets of \(V\) show that \(v\) is concave.  The zero graph of \(G\) is
therefore
\[
 x=\max\{c(y),v(y)\}.
\]
In view of \eqref{eq:four-state-zero-graph}, this would give
\[
 \psi=\min\{h,k\},
 \qquad
 h\coloneqq-\mu^{-1}c\ \text{concave},
 \qquad
 k\coloneqq-\mu^{-1}v\ \text{convex}.
\]
The set on which the convex branch \(k\) is active is an interval, since
\(h-k\) is concave.  On the interior of that interval the minimum agrees
with \(k\), so the slope jumps are upward; outside it the minimum agrees with
\(h\), so the slope jumps are downward.  At either endpoint, concavity of
\(h-k\) shows that the switch between the branches is also a downward jump.
Thus the slope jumps must occur in three consecutive blocks: downward,
upward, and downward, with any block possibly empty.  The actual pattern of
\(\psi\) is upward, downward, upward. This excludes the structural class
treated in \cite{CabreCaffarelli2003}.

The dual representation is impossible as well.  It would lead to
\(\psi=\max\{h,k\}\), with \(h\) concave and \(k\) convex.  A switch between
the two branches of such a maximum produces an upward, not a downward,
corner.  Hence the downward corner of \(\psi\) at \(y=1\) must be realized
by the concave branch on both sides.  Thus \(h(1)=2\), and the adjacent
slopes \(2\) and \(-2\) provide the supporting
lines
\[
 h(y)\le2y\quad(y<1),
 \qquad
 h(y)\le4-2y\quad(y>1).
\]
It follows that \(h<\psi\) on both tails, so the same convex branch \(k\)
must agree with \(\psi\) there.  Its left and right asymptotic slopes would
then be \(0\) and \(-1\), contradicting the monotonicity of the slopes of a
convex function.

There is a second obstruction.  The graph
\eqref{eq:four-state-zero-graph} has genuine corners.  If the same zero
equation were represented by a \(C^1\) uniformly elliptic operator
\(\widetilde F\), with ellipticity constants
\(\widetilde\lambda\) and \(\widetilde\Lambda\), then
\[
 \widetilde f(x,y)\coloneqq\widetilde F(M(x,y))
\]
would satisfy
\[
 \widetilde\lambda
 \le\partial_x\widetilde f(x,y)
 \le\widetilde\Lambda.
\]
The implicit function theorem would make its zero graph \(C^1\), a
contradiction.  Thus the differentiability assumption in the
Armstrong--Silvestre--Smart criterion
\cite[Theorem~1]{ArmstrongSilvestreSmart2012} cannot be recovered by changing the
representative of the equation.

\paragraph{The gain from centered tangents.}
The uncentered asymptotic profiles are much simpler than the original
operator.  Since
\[
 (T_tF_\mu)(M)
 =\operatorname{tr}M+\mu t\,
   \psi\!\left(\frac{\ell(M)}t\right),
\]
one has, locally uniformly in \(M\),
\[
 T_tF_\mu(M)\longrightarrow
 \operatorname{tr}M-\mu\bigl(\ell(M)\bigr)_+
 \qquad\text{as }t\downarrow0,
\]
and
\[
 T_tF_\mu(M)\longrightarrow
 \operatorname{tr}M+2\mu\bigl(\ell(M)\bigr)_+
 \qquad\text{as }t\uparrow\infty.
\]
Only the first limit is needed for the standard asymptotic estimates.  Its
concavity places \(F_\mu\) within the scope of 
\cite[Corollary~2]{SilvestreTeixeira2015}. By the
Evans--Krylov theorem, its equation has interior $C^{1,1}$ estimates and
therefore also meets the hypothesis of
\cite[Theorem~1.1]{PimentelTeixeira2016}. Consequently,
\[
 u\in C^{1,\alpha}_{\mathrm{loc}}\quad\text{for every }\alpha<1,
 \qquad
 u\in W^{2,p}_{\mathrm{loc}}\quad\text{for every }p>n.
\]
In particular, the standard differentiability theorem applied to
\(Du\in W^{1,p}_{\mathrm{loc}}\), with \(p>n\), shows that solutions have a
quadratic Peano expansion almost everywhere.  These estimates, however,
provide no continuity of the Hessian and do not identify an open regular
set.  The convex limit as \(t\uparrow\infty\) is a separate
favorable feature of the example; it is not used in these conclusions.

The centered tangents retain the missing information.  Away from the three
switching levels they are linear.  At the levels \(0\) and \(2\), the slope
increases and the tangent is convex; at the level \(1\), the slope decreases
and the tangent is concave.  Thus every centered tangent is covered by
Evans--Krylov, although the original four-state equation is not.
\Cref{cor:finite-switching} gives
\[
 u\in C^{2,\beta}_{\mathrm{loc}}(\Reg(u))
 \qquad(0<\beta<\alpha_{\mathrm{EK}}),
\]
and the complement of \(\Reg(u)\) has universal positive codimension.

For a concrete three-dimensional example, take \(\mu=2/5\). The equation is
\begin{equation}
 \Delta u+\frac25\psi(u_{11}-u_{22})=0.
 \label{eq:three-dimensional-isaacs}
\end{equation}
Its ellipticity constants are \(1/5\) and \(9/5\).  Thus the asymptotic
theory gives nearly Lipschitz gradients and arbitrarily high Hessian
integrability, while the centered tangent mechanism yields partial H\"older
continuity of the Hessian.

\section*{Acknowledgments}
A.S. acknowledges support
from King Abdullah University of Science and Technology (KAUST) under Award No. ORFS-
CRG12-2024-6430. E.V.T. gratefully acknowledges support from the Grayce B. Kerr Chair funds
at Oklahoma State University.

\section*{Declarations}

\subsection*{Data availability statement} No data has been produced in the elaboration of this work.

\end{document}